\pdfoutput=1
\RequirePackage{ifpdf}
\ifpdf 
\documentclass[pdftex]{sigma}
\else
\documentclass{sigma}
\fi

\numberwithin{equation}{section}

\newtheorem{Theorem}{Theorem}[section]
\newtheorem*{Theorem*}{Theorem}

\newtheorem{Lemma}[Theorem]{Lemma}
\newtheorem{Proposition}[Theorem]{Proposition}
 { \theoremstyle{definition}
\newtheorem{Definition}[Theorem]{Definition}

\newtheorem{Example}[Theorem]{Example}
\newtheorem{Remark}[Theorem]{Remark} }

\begin{document}
\allowdisplaybreaks

\newcommand{\arXivNumber}{2108.02087}

\renewcommand{\PaperNumber}{097}

\FirstPageHeading

\ShortArticleName{Weil Classes and Decomposable Abelian Fourfolds}

\ArticleName{Weil Classes and Decomposable Abelian Fourfolds}

\Author{Bert VAN GEEMEN}

\AuthorNameForHeading{B.~van Geemen}

\Address{Dipartimento di Matematica, Universit\`a di Milano, Via Saldini 50, I-20133 Milano, Italy}
\Email{\href{mailto:lambertus.vangeemen@unimi.it}{lambertus.vangeemen@unimi.it}}

\ArticleDates{Received May 11, 2022, in final form December 06, 2022; Published online December 13, 2022}

\Abstract{We determine which codimension two Hodge classes on $J\times J$, where $J$ is a~general abelian surface, deform to Hodge classes on a family of abelian fourfolds of Weil type. If a Hodge class deforms, there is in general a unique such family. We show how to determine the imaginary quadratic field acting on the fourfolds of Weil type in this family as well as their polarization. There are Hodge classes that may deform to more than one family. We relate these to Markman's Cayley classes.}


\Classification{14C30; 14C25; 14K20}

\section{Introduction}
The Hodge conjecture asserts that for any smooth complex projective variety $X$ and any
non-negative integer $p$ the vector space $B^p(X):=H^{2p}(X,{\bf Q})\cap H^{p,p}(X)$
of Hodge classes is spanned by the classes of algebraic cycles.

From a study of the Mumford Tate groups of abelian fourfolds, Moonen and Zarhin~\cite{MZ1,MZ2},
obtained the result that if the Hodge conjecture holds for abelian fourfolds of Weil type,
then it holds for all abelian fourfolds. An abelian fourfold is of Weil type if it has
an imaginary quadratic field in its endomorphism algebra
and each non-rational element in that field
has two eigenvalues on $H^{1,0}$ with the same multiplicity.
Such fourfolds are parametrized by four-dimensional subvarieties in the moduli space
of all abelian fourfolds with a fixed polarization.

The first results on the Hodge conjecture for fourfolds of Weil type were obtained
by Chad Schoen \cite{schoenhodge, Sadd}.
The paper \cite{schoen} provides an alternative
construction of algebraic cycles representing `exceptional' Hodge classes in
one of the cases considered in~\cite{schoenhodge}
and might have much wider applications.
Starting from a certain reducible surface $S\subset J\times J$,
where $J$ is the Jacobian of a general genus two curve, he shows that $S$ deforms
in a specific four-dimensional family of abelian fourfolds of Weil type.
This verifies the Hodge conjecture for these abelian fourfolds.
The general deformation of $S$, now called a Schoen surface,
was studied in \cite{CMLR, RRS} and a rather explicit description
of these surfaces can be found there.

In this paper we consider more generally the question
of which Hodge classes in $B^2(J\times J)$
can be deformed to an abelian fourfold of Weil type,
and which is the imaginary quadratic field acting on such a fourfold.
We provide answers to these questions in Section~\ref{limHC}.
In particular, such a Hodge class must be contained
in a determinantal cubic fourfold $Z\subset {\bf P} B^2(J\times J)\cong {\bf P}^5$
which is the secant variety of a Veronese surface.
This Veronese surface is the singular locus of~$Z$. A~Hodge class in $Z$, but not in $\operatorname{Sing}(Z)$,
will deform in at most one family. A~Hodge class in $\operatorname{Sing}(Z)$ may
deform to infinitely many families with distinct quadratic fields.
This rather curious fact is also the reason we restrict ourselves to fourfolds in the last sections.
In case $\dim J=n>2$ the deformation of a Hodge class to a family of Weil type, if it exists,
seems to be unique. However, a good description of the subvariety of ${\bf P} B^n(J\times J)$ of classes
that admit such deformations is much harder to find.

In his recent paper \cite{Markman}, E.~Markman
proved the Hodge conjecture for the
abelian fourfolds of Weil type,
with a polarization having trivial discriminant, using methods
from deformation theory of hyperk\"ahler manifolds. He also uses specific Hodge classes,
called Cayley classes, that deform in Weil type families for any imaginary quadratic field.
We briefly discuss the relation with our results in Section~\ref{cayleyclasses}.

In Section \ref{cycles}, we discuss the Hodge classes defined by certain algebraic
cycles from the papers~\cite{schoen} and~\cite{vG}.
In Section~\ref{secdis}, we show that the discriminants of the polarized abelian fourfolds
of Weil type we consider are trivial.

\section{Hodge classes on abelian fourfolds of Weil type}\label{ab4w}

\subsection{Abelian varieties of Weil type}
Let $A$ be an abelian variety and let $K={\bf Q}\big(\sqrt{-d}\big)$, with $d\in{\bf Z}_{>0}$,
be an imaginary quadratic field. An abelian variety of Weil type (with field
$K$) is a pair $(A,K)$, where $A$ is an abelian variety and $K\hookrightarrow \operatorname{End}_{\bf Q}(A)$
is a subalgebra, such that for all $k\in K$
the endomorphism of $T_0A$
defined by the differential of $k=a+b\sqrt{-d}\in K$
has eigenvalues $k=a+b\sqrt{-d}$ and $\bar{k}=a-b\sqrt{-d}$ with the same multiplicity.
Equivalently, the eigenvalues of any $k\in K$ on $H^{1,0}$ have the same multiplicity.
In particular, if $(A,K)$ is of Weil type, then $\dim A$ is even.

Given an abelian variety of Weil type $(A,K)$,
there exists a polarization $\omega_K\in B^1(A)$ on $A$
such that for all $k\in K$ one has
\[
k^*\omega_K = \operatorname{Nm}(k)\omega_K,\qquad \operatorname{Nm}(k) = k\bar{k},
\]
where $\operatorname{Nm}(k)$ is the norm of $k\in K$ (see \cite[Lemma 5.2.1]{hcav}).
We call such a 2-form a polarization of Weil type and
$(A,K,\omega_K)$ is called a polarized abelian variety of Weil type.

\subsection{The Weil classes}
For a general (in particular $A$ is not decomposable) abelian variety of Weil type
$(A,K)$ with dimension $2n$, the spaces of Hodge classes have dimensions
\cite{Weil} (see also \cite[Theorem~6.12]{hcav}):
\[
\dim B^p(A) = 1,\quad p\ne n,\qquad \dim B^n(A) = 3.
\]
Since $\dim B^1(A)=1$, any $\omega\in B^1(A)$, $\omega\ne 0$,
defines (up to sign) a polarization on $A$ which will be of Weil type.
Writing $\omega_K:=\omega$, the $p$-fold exterior product of this class with itself,
denoted by $\omega_K^p\in \wedge^pB^1(A)\subset B^p(A)$, is non-zero
for any $p\leq 2n$.
Using the action of the multiplicative group $K^\times:=K-\{0\}$ on $H^{2n}(A,{\bf Q})$,
there is a natural two-dimensional subspace $W_K$
in $B^n(A)$ which is a complement of ${\bf Q}\omega_K^n$:
\[
B^n(A) = {\bf Q}\omega_K^n \oplus W_K,\qquad \dim W_K = 2.
\]

To define $W_K$ we use that for all $k\in K^\times$
the eigenvalues of $k^*$ on $H^1(A,{\bf Q})$ are $k$, $\bar{k}$,
each with multiplicity $2n$. Let
\[
H^1(A,K)=H^1(A,{\bf Q})\otimes_{\bf Q} K=V_+\oplus V_-,\qquad k^* = \operatorname{diag}\big(k,\bar{k}\big)
\]
be the direct sum decomposition into $2n$-dimensional, conjugate, eigenspaces
$V_\pm$ of $K^\times$.
The eigenvalues of $k^*$ on $H^{1,0}(A)$ (and on $H^{0,1}(A)$) are $k$, $\bar{k}$ each with
multiplicity $n$,
so $V_+$, $V_-$ both have Hodge numbers $h^{1,0}=h^{0,1}=n$.
Since $H^{2n}(A,{\bf Q})=\wedge^{2n}H^1(A,{\bf Q})$, we get
\[
H^{2n}(A,K) = \oplus_{m=0}^{2n} \big({\wedge}^{2n-m}V_+\big)\otimes\big({\wedge}^mV_-\big),
\]
and the eigenvalues of $k^*$ on the $m$-th summand are $k^{2n-m}\bar{k}^m$.
The one-dimensional subspaces $\wedge^{2n}V_+$ and $\wedge^{2n}V_-$
are of Hodge type $(n,n)$
and their direct sum is defined over ${\bf Q}$ since they are conjugate.
That is, there is two-dimensional subspace $W_K\subset H^{2n}(A,{\bf Q})$ such that
\[
W_K\otimes_{\bf Q} K = \wedge^{2n}V_+ \oplus \wedge^{2n}V_-,\qquad
W_K\subset H^{2n}(A,{\bf Q})\cap H^{n,n}(A) = B^n(A).
\]

We give another definition of $W_K$.
Notice that if $k\in K^\times$, and $k^{2n}=a+b\sqrt{-d}\in K$,
with $a,b\in{\bf Q}$, then $k^{2n}$, $\bar{k}^{2n}$ are the roots of the polynomial
$P(T):=T^2-2aT+a^2+db^2\in{\bf Q}[T]$.
Let $W_K$ be the kernel of the linear map $P(k^*)\in \operatorname{End}(H^{2n}(A,{\bf Q}))$,
for any general $k\in K^\times$ (i.e., such that
$k^{2n},k^{2n-1}\bar{k},\dots,\bar{k}^{2n}$ are distinct).
Then $W_K\otimes_{\bf Q} K$ is the span of the eigenspaces of~$k^*$
with eigenvalues $k^{2n}$, $\bar{k}^{2n}$.
Hence $W_K\otimes_{\bf Q} K=\wedge^{2n}V_+\oplus\wedge^{2n}V_-$ as before, so
$\dim_{\bf Q} W_K=2$ and the classes in $W_K$ are of Hodge type~$(n,n)$.

\begin{Definition} \label{KB1B2}
Let $(A,K,\omega_K)$ be a $2n$-dimensional abelian variety of Weil type.
The two-dimensional subspace $W_K=W_{K,A}$ of $B^n(A)$ is called the space of Weil classes
of $(A,K)$.

This subspace is characterized by $W_K\otimes_{\bf Q} K=W_{K,+}\oplus W_{K,-}$ where
$W_{K,+}, W_{K,-}$ are the (one-dimensional) eigenspaces of the $K^\times$ action
on $H^{2n}(A,K)$, with eigenvalues $k^{2n}$, $\bar{k}^{2n}$ respectively, for all $k\in K^\times$
where $K^\times:=K-\{0\}$ be the multiplicative group of $K$.

The space of Hodge classes of $(A,K,\omega_K)$ is
the three-dimensional subspace $B^n_{K,A}$ of the space of Hodge classes $B^n(A)$
defined as
\[
B^n_{K,A} := {\bf Q}\omega_K^n \oplus W_K.
\]
\end{Definition}

\subsection{Families of abelian varieties of Weil type}\label{famwt}
Let $(A,K,\omega_K)$ be a polarized abelian variety of Weil type of dimension $2n$.
Then $A$ is a~complex torus, $A={\bf C}^{2n}/\Lambda$.
Intrinsically, $\Lambda=\pi_1(A)=H_1(A,{\bf Z})$ and ${\bf C}^{2n}=\Lambda\otimes_{\bf Z}{\bf R}=H_1(A,{\bf R})$,
with a~certain complex structure $I\in \operatorname{End}(H_1(A,{\bf R}))$, so $I^2=-{\rm Id}$.
This complex structure induces one on the dual vector space $H^1(A,{\bf R})$ which has eigenspaces
$H^{1,0}(A)$, $H^{0,1}(A)$. These again determine the Hodge decomposition on all $H^k(A)$,
since $H^k(A,{\bf Z})=\wedge^kH^1(A,{\bf Z})$.
The embedding $K\hookrightarrow \operatorname{End}_{\bf Q}(A)$ determines (and is determined by)
an embedding ${K\hookrightarrow \operatorname{End}_{\bf Q}(\Lambda\otimes_{\bf Z}{\bf Q})}$ and the polarization
$\omega_K\in H^2(A,{\bf Q})$ is an alternating 2-form on $\Lambda\otimes_{\bf Z}{\bf Q}$.

To deform $(A,K,\omega_K)$ to other abelian varieties of Weil type, one changes the complex
structure $I$ in such a way that $K$ still acts by ${\bf C}$-linear maps on $\Lambda\otimes_{\bf Z}{\bf R}$
and such that $\omega_K$ still satisfies the Riemann conditions for a polarization
(cf.\ \cite[Section 5]{hcav}). In this way one obtains
an $n^2$-dimensional family of abelian varieties of Weil type parametrized by
a bounded Hermitian domain.

For each $(A',K,\omega_K)$ in this family there is a natural identification
$H^k(A',{\bf Q})=H^k(A,{\bf Q})$ since only the complex structure varies.
As the $K$-action is fixed, both $\omega_K$
(and thus also~$\omega_K^n$) and the space of Weil cycles is fixed
($W_{K,A'}=W_{K,A}$) as is the spaces of Hodge classes of $(A',K,\omega_K)$
($B^n_{K,A'}=B^n_{K,A}={\bf Q}\omega_K^n\oplus W_{K,A}$)
under the identification $H^{2n}(A',{\bf Q})=H^{2n}(A,{\bf Q})$.

In this paper we consider, for a decomposable abelian $2n$-fold $J\times J$, with $J$ general,
and all imaginary quadratic subfields $K\subset \operatorname{End}_{\bf Q}\big(J^2\big)$, the corresponding
three-dimensional spaces $B^n_{K,J^2}\subset B^n\big(J^2\big)$. In case $n=2$, we will see that the
intersection between these subspaces is non-trivial.

\section{Weil classes for decomposable abelian varieties}\label{decaf}

\subsection{General polarized abelian varieties}\label{JJ}
We fix a general abelian variety $J$ of dimension $n$. In particular $J$ has
endomorphism algebra $\operatorname{End}_{\bf Q}(J)={\bf Q}$ and $\dim B^1(J)=1$ (so $J$
has a unique polarization up to scalar multiple).

Let $\omega_J\in B^1(J)$ be the class of an ample divisor, we denote
by $(J,\omega_J)$ the corresponding polarized abelian variety.
We choose a basis of 1-forms ${\rm d} t_j$, $j=1,\dots,2n$, of $H^1(J,{\bf Q})$ such that
\[
\omega_J={\rm d} t_1\wedge{\rm d} t_2+ \dots + {\rm d} t_{2n-1}\wedge {\rm d} t_{2n}.
\]
Then we obtain a basis of
$H^1\big(J^2,{\bf Q}\big)\cong H^1(J,{\bf Q})^{\oplus 2}$ given by the
${\rm d} x_j=\pi_1^*{\rm d} t_j$, ${\rm d} y_j:=\pi_2^*{\rm d} t_j$, $j=1,\dots,2n$, where the
$\pi_i\colon J^2 \rightarrow J$, $i=1,2$, are the projections on the factors.

The following proposition determines the dimensions, as well as bases, for the spaces of
Hodge classes for the product variety $J\times J$.

\begin{Proposition} \label{B2}
For a general abelian $n$-fold $J$ and for $1\leq p\leq n$ we have
\[
\dim B^1\big(J^2\big) = 3,\qquad B^p\big(J^2\big) \cong \operatorname{Sym}^pB^1\big(J^2\big),
\qquad \text{so}\quad\dim B^{p}\big(J^2\big) = \binom{p+2}{2}.
\]
A basis of $B^1(J\times J)$ is given by the following three $2$-forms
\begin{gather*}
\omega_1={\rm d} x_1\wedge{\rm d} x_2 + \dots + {\rm d} x_{2n-1}\wedge {\rm d} x_{2n},\\
\omega_2={\rm d} y_1\wedge{\rm d} y_2 + \dots + {\rm d} y_{2n-1}\wedge {\rm d} y_{2n},\\
\omega_\sigma={\rm d} x_1\wedge{\rm d} y_2 - {\rm d} x_2\wedge {\rm d} y_1 + \dots +
 {\rm d} x_{2n-1}\wedge{\rm d} y_{2n}-{\rm d} x_{2n}\wedge {\rm d} y_{2n-1}.
\end{gather*}
A basis of $B^2(J\times J)$ is thus given by the following six $4$-forms
$($where we write $\omega\theta$ for $\omega\wedge\theta)$:
\[
\omega_1^2, \quad\omega_1\omega_2, \quad
\omega_1\omega_\sigma,\quad
\omega_2^2,\quad
\omega_2\omega_\sigma,\quad \omega_\sigma^2\qquad
\big(\!\!\in B^2\big(J^2\big)\big).
\]
\end{Proposition}

\begin{proof}Since $J$ is general, the Hodge group, also known as the special Mumford--Tate group,
of $V:=H^1(J,{\bf Q})$ is ${\rm Sp}(V)\cong {\rm Sp}_{2n}$, the symplectic group of $\omega_J$.
The spaces of Hodge classes $B^p\big(J^2\big)$ are obtained as the invariants (for the diagonal action):
\[
B^p\big(J^2\big) = \big({\wedge}^{2p}\big(V^{\oplus 2}\big)\big)^{{\rm Sp}(V)} =
\oplus_{k=0}^{2p} \big(\big({\wedge}^kV\big)\otimes\big({\wedge}^{2p-k}V\big)\big)^{{\rm Sp}(V)}.
\]
Interchanging $\wedge^kV$ and $\wedge^{2p-k}V$ shows that we only need to find the
invariants for $k\leq p$.
The symplectic form $\omega_J$ induces a duality $\wedge^kV\cong\big({\wedge}^kV\big)^*$ and thus
\[
\big(\big({\wedge}^kV\big)\otimes\big({\wedge}^{2p-k}V\big)\big)^{{\rm Sp}(V)} \cong
\operatorname{Hom}_{{\rm Sp}(V)}\big({\wedge}^kV,\wedge^{2p-k}V\big).
\]

For each $k$, with $0\leq k\leq n$,
the ${\rm Sp}(V)$-representation $\wedge^kV$ is a direct sum of $k/2+1$, $(k+1)/2$,
with~$k$ even and~$k$ odd respectively, irreducible ${\rm Sp}(V)$-representations,
each with multiplicity~$1$.
These representation are denoted by $V^{(k-2i)}$, $0\leq i\leq k/2$ in \cite[Theorem~17.5]{FH} and
$\wedge^kV\cong V^{(k)}\oplus\wedge^{k-2}V$.
In this way one obtains the Lefschetz decomposition \cite[p.~122]{GH} of $H^k(J,{\bf Q})=\wedge^kV$.

Notice that the number of these representations depends only on $k$ and not on~$n$.
We denote by $d_k$ this number of irreducible subrepresentations of~$\wedge^kV$,
so $d_k=1,1,2,2,3,3,4,\dots$ for $k=0,1,2,3,4,5,6,\dots$.

If $k\leq p$, $\wedge^kV$ is isomorphic to the ${\rm Sp}(V)$-subrepresentation
$\omega_J^{p-k}\wedge \big({\wedge}^kV\big)\subset \wedge^{2p-k}V$.
By Schur's lemma, the dimension of $\operatorname{Hom}_{{\rm Sp}(V)}\big({\wedge}^kV,\wedge^{2p-k}V\big)$ is then equal to
the number of irreducible ${\rm Sp}(V)$-representations in $\wedge^kV$.
In particular $\dim B^1\big(J^2\big)=1+1+1=3$ and $\dim B^2\big(J^2\big)=1+1+2+1+1=6$.
More generally, for $0\leq k\leq p$,
\[
\dim B^p\big(J^2\big) = \sum_{k=0}^{2p} d_k =
2\sum_{k=0}^{p-1}d_k\;+ d_{p} = \dim B^{p-1}\big(J_{n}^2\big) + d_{p-1} + d_p.
\]
Since $d_{p-1}+d_p=p+1$ and $\dim B^0\big(J^2\big)=1$, we find $\dim B^p\big(J^2\big)=\binom{p+2}{2}$.

According to Hazama \cite[Proposition 1.6]{Haz}
the ring of Hodge classes $B^*\big(J^2\big)=\oplus B^p\big(J^2\big)$ is generated by $B^1\big(J^2\big)$,
so the cup product $\operatorname{Sym}^pB^1\big(J^2\big)\rightarrow B^p\big(J^2\big)$ is surjective. Comparing dimensions we
get $\operatorname{Sym}^pB^1\big(J^2\big)\cong B^p\big(J^2\big)$.

Notice that $\omega_j=\pi_j^*\omega_J$, hence they are in $B^1\big(J^2\big)$. Using the addition map
\[
\sigma \colon \ J\times J \longrightarrow J,\qquad
(x,y)\longmapsto t=x+y,
\]
one finds $\omega_\sigma\in B^1\big(J^2\big)$ from
\[
\sigma^*\omega_J =
\omega_1+\omega_2+\omega_\sigma\qquad\big(\!\!\in B^1\big(J^2\big)\big).
\]
Since $\dim B^1\big(J^2\big)=3$, these three 2-forms are a basis of $B^1\big(J^2\big)$.
The six 4-forms are obviously a basis of $\operatorname{Sym}^2B^1\big(J^2\big)= B^2\big(J^2\big)$.
\end{proof}

\subsection[The GL\_2-action on the Hodge cycles]{The $\boldsymbol{{\rm GL}_2}$-action on the Hodge cycles}\label{GL2Hodge}

As $\operatorname{End}_{\bf Q}(J)={\bf Q}$, we have $\operatorname{End}_{\bf Q}\big(J^2\big)=M_2({\bf Q})$ (we write $(x,y)\in J^2$ as
a column vector).
We will first consider the induced action of ${\rm GL}_2$ on the Hodge cycles $B^p\big(J^2\big)$.
Since for pull-backs, $(gh)^*=h^*g^*$ where $g,h\in M_2({\bf Q})$, to get a representation
of ${\rm GL}_2$ on the cohomology we define
$g\cdot \omega:=({}^tg)^*\omega$ for $\omega\in H^*\big(J^2,{\bf Q}\big)$.
We denote the standard two-dimensional representation of~${\rm GL}_2$ by $V_1$. Then we
have $H^1\big(J^2,{\bf Q}\big)\cong H^1(J,{\bf Q})\otimes V_1$ and the~${\rm GL}_2$ action is on the second factor.

Recall that ${\rm SL}_2$ has a unique irreducible representation of dimension $m$ for any
$m\geq 1$. There is an isomorphism $\operatorname{Sym}^2V_1\cong B^1\big(J^2\big)$ which simplifies some of the arguments in
the rest of the paper.

\begin{Proposition} \label{GLB}\label{decGLB}
The vector spaces $B^p\big(J^2\big)$ are ${\rm GL}_2$-representations
in which the center ${\bf Q}^\times$ of ${\rm GL}_2$ acts by
scalar multiplication by $t^{2p}$ on $B^p\big(J^2\big)$.
As representation for the subgroup ${\rm SL}_2$,
the space $B^1\big(J^2\big)$ is the irreducible three-dimensional
representation of ${\rm SL}_2$. Moreover, if~$e_+$,~$e_-$ are the standard basis of $V_1$, then there
is an isomorphism of ${\rm GL}_2$-representations
\[
\operatorname{Sym}^2V_1 \stackrel{\cong}{\longrightarrow} B^1\big(J^2\big),
\qquad \begin{cases} e_+^2 \longmapsto \omega_1,\\
 e_+e_- \longmapsto \tfrac{1}{2}\omega_\sigma,\\
 e_-^2 \longmapsto \omega_2.
 \end{cases}
\]
The space $B^2\big(J^2\big)=\operatorname{Sym}^2B^1\big(J^2\big)$ is the direct sum of
the two irreducible ${\rm SL}_2$-subrepresen\-ta\-tions~$V_0$, $V_4$ of $B^2\big(J^2\big)$
of dimension $1$ and $5$:
\[
V_0 := \big\langle \Omega := 4\omega_1\omega_2 - \omega_\sigma^2 \big\rangle,
\qquad
V_4 =
\big\langle \omega_1^2,\ \omega_1\omega_\sigma,
2\omega_1\omega_2+\omega_\sigma^2,
\omega_2\omega_\sigma, \omega_2^2 \big\rangle.
\]
\end{Proposition}

\begin{proof}
Since $t\in{\bf Q}^\times$ acts multiplication by $t$ on $H^1\big(J^2,{\bf Q}\big)$ and
$H^k\big(J^2,{\bf Q}\big)=\wedge^kH^1\big(J^2,{\bf Q}\big)$, it acts as $t^{2p}$ on $B^p\big(J^2\big)$.

The diagonal matrix $\operatorname{diag}\big(t,t^{-1}\big)\in {\rm SL}_2$ acts as $\operatorname{diag}\big(t^{2},1,t^{-2}\big)$
on the basis $\omega_1$, $\omega_\sigma$, $\omega_2$ of $B^1\big(J^2\big)$ which proves that $B^1\big(J^2\big)$
is the three-dimensional irreducible representation of~${\rm SL}_2$.
This representation is also isomorphic to $\operatorname{Sym}^2V_1$. Using the eigenvectors of the
diagonal matrices we see that, up to scalar multiples,
the isomorphism is as in the proposition.
We fix the isomorphism by imposing that $e_+^2\mapsto \omega_1$. We define
\[
S := \begin{pmatrix}\hphantom{-}0&1\\-1&0\end{pmatrix},\qquad T := \begin{pmatrix}1&1\\0&1\end{pmatrix}
\qquad\big(\!\!\in {\rm SL}_2).
\]
Then $Se_+=e_-$ and thus $Se_+^2=e_-^2$, therefore $e_-^2$ maps to $S\omega_1=\omega_2$.
Finally we have $Te_+e_-=e_+(e_++e_-)=e_+^2+e_+e_-$ and we know that $e_+e_-$ maps to
$a\omega_\sigma$ for some constant $a$. Thus $T$ maps $a\omega_\sigma$ to $\omega_1+a\omega_\sigma$.
An explicit computation shows that $({}^tT)^*\omega_\sigma=2\omega_1+\omega_\sigma$,
hence $a=1/2$.

From the basis of $B^2\big(J^2\big)$ given in Proposition \ref{B2} one sees that
$B^2\big(J^2\big)=\operatorname{Sym}^2B^1\big(J^2\big)$
and thus $B^2\big(J^2\big)$ is the sum of the irreducible five and one-dimensional representations of ${\rm SL}_2$.
With standard Lie algebra computations one finds the explicit decomposition.
\end{proof}

\subsection[Homomorphisms K hookrightarrow End\_Q(J\^{}2)]{Homomorphisms $\boldsymbol{K\hookrightarrow \operatorname{End}_{\bf Q}\big(J^2\big)}$}\label{KEnd}

For any imaginary quadratic field $K$,
there are embeddings $K\hookrightarrow \operatorname{End}_{\bf Q}\big(J^2\big)\cong M_2({\bf Q})$,
here an embedding is an injective homomorphism of rings. For example,
\[
\phi\colon \ K \hookrightarrow M_2({\bf Q}),\qquad \phi\big(a+b\sqrt{-d}\big) =
\begin{pmatrix}a &-db\\b&a\end{pmatrix}.
\]
Such an embedding gives a homomorphism (of algebraic groups defined over ${\bf Q}$)
$K^\times \rightarrow {\rm GL}_2$.
The image of $K^\times$ is a diagonalizable two-dimensional abelian subgroup of ${\rm GL}_2$,
so it is a Cartan subgroup of ${\rm GL}_2$.

In particular, $K$ has two eigenvectors $v_\pm\in V_1\otimes_{\bf Q} K$ on which it acts
as $k\cdot v_+=kv_+$ and $k\cdot v_-=\bar{k} v_-$. Conversely, given a basis of conjugated
vectors $v_\pm$ of $V_1\otimes_{\bf Q} K$, one finds an embedding $K\hookrightarrow M_2({\bf Q})$.
Such a pair of vectors can be chosen as $v_\pm=(\alpha_\pm,1)\in K^2$ with
$\overline{\alpha_+}=\alpha_-$. We can thus parametrize all embeddings of $K$ by $\alpha_+\in K-{\bf Q}$.

\subsection{Cartan subalgebras}\label{csa}
We identify the Lie algebra of ${\rm GL}_2$ with $M_2$. Then a subfield $K\subset M_2({\bf Q})$ is a
Cartan subalgebra of $M_2$, that is a diagonalizable two-dimensional subalgebra.
More generally, we will consider Cartan subalgebras, still denoted by $K$, in $M_2({\bf C})$.
Such a Cartan subalgebra determines, and is determined by, a basis (of eigenvectors) of
$V_1\otimes_{\bf Q}{\bf C}$. Only rather special bases will correspond to imaginary quadratic subfields
of $M_2({\bf Q})$, but it is still convenient to consider all of them as we do from now on.
The complexification of the space of Hodge cycles
$B^p\big(J^2\big)$ is denoted simply by $B^p$:
\[
B^p := B^p(J\times J)\otimes_{\bf Q}{\bf C}.
\]
We refer to the eigenspaces and eigenvalues of the multiplicative group $K^\times$ of $K$
as the eigen\-spaces and eigenvalues of $K$.

The following lemma collects some explicit computations on the differential forms that are
eigenvectors of a Cartan subalgebra $K\hookrightarrow M_2({\bf C})$,
where $K$ is determined by a pair of independent eigenvectors $v_\pm \in {\bf C}^2$.

\begin{Lemma} \label{lemKB}
Let $v_\pm=(\alpha_\pm,1) \in V_1\otimes_{\bf Q}{\bf C}={\bf C}^2$
and let $K\hookrightarrow M_2({\bf C})$
be the Cartan subalgebra with these eigenvectors.
\begin{enumerate}\itemsep=0pt
\item[$1.$] A basis of eigenvectors of $K$ in $B^1$ is
\[
\omega_\pm := \alpha_\pm^2\omega_1 + \alpha_\pm\omega_\sigma + \omega_2,
\qquad\text{and}\qquad
\omega_K := 2\alpha_+\alpha_-\omega_1 + (\alpha_++\alpha_-)\omega_\sigma +
2\omega_2.
\]

\item[$2.$]
Let $K$ be an imaginary quadratic field. Then the eigenvalues of $K$ on $\omega_+$, $\omega_-$, $\omega_K$
are $k^2$, $\bar{k}^2$, $k\bar{k}$ respectively. Moreover, $\big(J^2,K,\omega_K\big)$
is an abelian variety of Weil type.
The space of Weil cycles $W_K\subset H^{2n}\big(J^2,{\bf Q}\big)$
of $\big(J^2,K,\omega_K\big)$ is spanned $($over $K)$ by $\omega_+^n$ and $\omega_-^n$:
\[
W_K\otimes_{\bf Q} K = K\omega_+^n \oplus K\omega_-^n,
\]
and $B^n_{K,J^2}={\bf Q} \omega_K^n\oplus W_K$ $($cf.\ Definition {\rm \ref{KB1B2})}.
\end{enumerate}
\end{Lemma}

\begin{proof}Writing $v_\pm=\alpha_\pm e_++e_-$ we get $v_\pm^2=\alpha_\pm^2e_+^2+2\alpha_\pm e_+e_-+e_-^2$,
which are eigenvectors of the $K^\times$ action on $\operatorname{Sym}^2V_1$.
The isomorphism $\operatorname{Sym}^2V_1\cong B^1\big(J^2\big)$ maps these to the 2-forms in this lemma. Similar for
$2v_+v_-$ which is a third eigenvector in $\operatorname{Sym}^2V_1$.

Let $k\in K^\times$, it has eigenvalues $k$, $\bar{k}$ on $v_+$, $v_-$, respectively,
hence the eigenvalues of $k$ on $B^1\big(J^2\big)\cong \operatorname{Sym}^2V_1$ are $k^2$, $\bar{k}^2$, $\operatorname{Nm}(k)$.
The $K^\times$ eigenspace in $B^1\big(J^2\big)$ with eigenvalue $\operatorname{Nm}(k)$ is thus one-dimensional
and therefore $\omega_K$ is unique up to scalar multiple, hence
it is (up to sign) an ample class (see \cite[Lemma~5.2]{vG}).
Restricting $\omega_K$ to the first factor of $J^2$ we conclude that~$\omega_K$ is in fact ample.
Thus $\omega_K^{2n}\neq 0$. So $\big(J^2,K,\omega_K\big)$
is an abelian variety of Weil type.

Since $k^*\omega_+ = k^{2}\omega_+$ we find
$k^*\omega^n_+=k^{2n}\omega_+^n$ and similarly
$k^*\omega^n_-=\bar{k}^{2n}\omega_-^n$.
As $\operatorname{Sym}^nB^1\big(J^2\big)\cong B^n\big(J^2\big)$, the forms $\omega_n^\pm$ are non-zero.
Definition~\ref{KB1B2} then shows that~$W_K$
is spanned by $\omega_+^{n}$ and~$\omega_-^n$.
\end{proof}

By Lemma~\ref{lemKB}\,(2),
any embedding $K\hookrightarrow \operatorname{End}_{\bf Q}\big(J^2\big)$ of an imaginary quadratic field
defines a polarized abelian variety, of dimension $2n$, of Weil type $\big(J^2,K,\omega_K\big)$,
which is unique up to multiplication of $\omega_K$ by a positive rational number.
In particular, the embedding determines an $n^2$-dimensional
family of abelian fourfolds of Weil type as in Section~\ref{famwt}
by changing the complex structure on $H_1\big(J^2,{\bf R}\big)$.
For any member $A$ in this family there is a natural identification of the
three-dimensional vector space of Hodge classes $B^n_{K,A}\subset H^{2n}(A,{\bf Q})$ of $(A,K,\omega_K)$
with $B^n_{K,J^2}$.

More generally, we define analogous complex subspaces in the $B^p$ (the complexification
of the $B^p\big(J^2\big)$) for any Cartan subalgebra.

\begin{Definition} \label{defBpK}
Let $K\hookrightarrow M_2({\bf C})$
be a Cartan subalgebra with eigenvectors $v_\pm=(\alpha_\pm,1) \in V_1\otimes_{\bf Q}{\bf C}={\bf C}^2$.
With the notation from Lemma~\ref{lemKB}\,(1)
we define a two-dimensional subspace of $B^n=B^n\big(J^2\big)\otimes_{\bf Q}{\bf C}$ by
\[
W_K := {\bf C} \omega_+^n \oplus {\bf C} \omega_-^n\qquad (\subset B^n),
\]
and we define a three-dimensional subspace $B^n_K$ of Hodge classes of $J^2$ by
\[
B^n_K := {\bf C}\omega_K^n \oplus W_K\qquad (\subset B^n).
\]
\end{Definition}

\subsection[The invariant conic in B\^{}1]{The invariant conic in $\boldsymbol{B^1}$}\label{invcon}

There is a curious (if well-known) geometrical relation between the three
points in ${\bf P} B^1\cong{\bf P}^2$, where $B^1:=B^1\big(J^2\big)\otimes_{\bf Q} {\bf C}$
determined by the $K^\times$-eigenvectors
$\omega_K, \omega_+, \omega_-\in B^1$.

The Veronese image of ${\bf P} V_1$, where~$V_1$ is the standard representation of ${\rm GL}_2$,
is a conic~$C_2$ in ${\bf P}\big({\operatorname{Sym}}^2 V_1\big)\cong{\bf P} B^1$. This conic is thus invariant under the action
of ${\rm GL}_2$ on ${\bf P} B^1$.
It contains the points $\omega_\pm=v^2_\pm$
where $v_\pm$ are eigenvectors of $K^\times$ acting on $V_1$.
From the Lemma~\ref{lemKB} one then finds:
\[
C_2 = \big\{[a\omega_1+b\omega_\sigma+c\omega_2] \in {\bf P} B^1\colon b^2-ac = 0 \big\}.
\]

Let $p\in {\bf P} B^1$ be the point of intersection of the tangent lines to $C_2$ in the
two points defined~$\omega_\pm$. Since these two points are fixed by the action of $K^\times$
on $B^1$, also the point $p$ is fixed by $K^\times$. Since there are only three fixed
points in $B^1$ by Lemma \ref{lemKB} it follows that $p$ must be the point
defined by~$\omega_K$.

Conversely, given $\omega_K$, one finds $\omega_\pm$ (up to permutation) as the
points of intersection of the (two) lines through $p$ which are tangent to $C_2$
with $C_2$ itself.

\subsection{The invariant conic and Pfaffians} \label{pfaff}
This invariant conic also has another description in terms of the two-forms in
$B^1=B^1\big(J^2\big)\otimes_{\bf Q}{\bf C}$.
The Pfaffian $\operatorname{Pfaff}(\theta)\in{\bf C}$ of a two-form $\theta\in B^1$
is defined by
\[
\wedge^{2n}\theta =
((2n)!)\operatorname{Pfaff}(\theta){\rm d} x_1\wedge\dots\wedge{\rm d} y_{2n}\qquad\big(\!\!\in H^{4n}\big(J^2,{\bf C}\big) \cong {\bf C}\big).
\]
Since the Pfaffian is essentially invariant under ${\rm GL}_2$, its zero locus is a union of
${\rm GL}_2$-orbits in~$B^1$. There are only two ${\rm GL}_2({\bf C})$-orbits on ${\bf P} B^1$
and, since $\operatorname{Pfaff}(\omega_\sigma)\neq 0$ but
$\operatorname{Pfaff}(\omega_1)=0$, for $\theta\in B^1$ we must have $\operatorname{Pfaff}(\theta)=e\big(b^2-ac\big)^n$
for some non-zero constant $e$.

\section[The Hodge classes of (J\^{}2,K) when n=dim J=2]{The Hodge classes of $\boldsymbol{\big(J^2,K\big)}$ when $\boldsymbol{n=\dim J=2}$}\label{limHC}

\subsection[The Hodge classes B\_K\^{}2]{The Hodge classes $\boldsymbol{B_K^2}$}\label{cubic4}

For a general abelian surface $J$,
an embedding of an imaginary quadratic field $K\hookrightarrow \operatorname{End}_{\bf Q}\big(J^2\big)=M_2({\bf Q})$
determines an abelian fourfold of Weil type $\big(J^2,K,\omega_K\big)$.
The space of Hodge classes of $\big(J^2,K,\omega_K\big)$ is the three-dimensional
subspace of $B^2\big(J^2\big)\subset H^4\big(J^2,{\bf Q}\big)$, see Definition \ref{KB1B2},
\[
B^2_{K,J^2} := {\bf Q} \omega_K^2 \oplus W_{K,J^2},\qquad \text{with}\quad
W_{K,J^2}\otimes_{\bf Q} K = K\omega_+^2 \oplus K\omega_-^2.
\]
As observed in Section~\ref{famwt},
$\big(J^2,K,\omega_K\big)$ determines
a four-dimensional family of abelian varieties of Weil type by changing the complex structure
on $H_1\big(J^2,{\bf R}\big)$.
For any $A$ in this family the three-dimensional space of Hodge classes~$B^2_{K,A}$
of $(A,K,\omega_K)$ is the same as $B^2_{K,J^2}$.

To study the subspaces of Hodge classes $B^2_K$ as $K\subset M_2({\bf Q})$
varies over the imaginary quadratic subfields,
it is convenient to complexify and projectivize. In that way,
as $K$ varies over the Cartan subalgebras of $M_2({\bf C})$,
we obtain a two-dimensional family of complex projective planes ${\bf P} B^2_K$,
with $B^2_K$ as in Definition~\ref{defBpK},
in the complex projective space ${\bf P} B^2={\bf P}^5$.

First of all we show that any two ${\bf P} B^2_K$'s have a non-empty intersection.

\begin{Proposition} \label{conicCK}
Given two distinct Cartan subalgebras $K_1$, $K_2$
of $M_2({\bf C})$, the intersection of their spaces of Hodge classes $B^2_{K_1}\cap B^2_{K_2}
\subset B^2$ is one-dimensional and thus ${\bf P} B^2_{K_1}\cap {\bf P} B^2_{K_2}$ is a~point.

Let $K$ be a Cartan subalgebra,
then the union of the points of intersection ${\bf P} B^2_K\cap {\bf P} B^2_{K'}$,
where~$K'$ runs over the Cartan subalgebra's distinct from $K$, is a conic, denoted by~$C_K$,
in~${\bf P} B^2_K$.
\end{Proposition}

\begin{proof}The action of ${\rm GL}_2({\bf C})$ on the Cartan subalgebras by conjugation is transitive,
since any element in ${\rm GL}_2({\bf C})$ that maps the eigenvectors of~$K_1$ to those of~$K_2$
will conjugate~$K_1$ to~$K_2$.
So we may assume that $K$ is the Cartan subalgebra of diagonal matrices.
The corresponding eigenvectors in~$B^1$ are
$\theta_+:=\omega_1$, $\theta_-:=\omega_2$
and $\theta_K:=\omega_\sigma$ (where we wrote $\theta$ rather than $\omega$),
so they are the standard basis of~$B^1$.
Therefore~$B^2_K$ is spanned by $\omega_1^2$, $\omega_\sigma^2$, $\omega_2^2$
and an element in~$B^2$ lies in~$B^2_K$ iff in the standard basis of~$B^2$ (given in
Proposition \ref{B2}) the coefficients of
$\omega_1\omega_\sigma$, $\omega_2\omega_\sigma$, $\omega_1\omega_2$ are zero.

For the Cartan subalgebra $K'$, $K'\neq K$, determined by $v_\pm:=(\alpha_\pm,\beta_\pm)\in{\bf C}^2$,
the eigenforms $\omega_\pm$, $\omega_{K'}$ are given in Lemma \ref{lemKB}
(we assumed there that $\beta_\pm=1$ but it is easy to homogenize the expressions).
Computing their squares, one finds that the $3\times 3$ matrix
of their coefficients of
$\omega_1\omega_\sigma$, $\omega_2\omega_\sigma$, $\omega_1\omega_2$ has rank two.
It is then easy to find the intersection:
\[
B^2_K\cap B^2_{K'} =
\big\langle 2\alpha_-^2\beta_+^2\omega_1^2 -
\alpha_+\alpha_-\beta_+\beta_-\omega_\sigma^2 + 2\alpha_+^2\beta_-^2\omega_2^2 \big\rangle.
\]
In particular, the corresponding point $a\omega_1^2+b\omega_\sigma^2+c\omega_2^2\in {\bf P} B_K^2$
lies on the conic $C_K$ defined by $4b^2=ac$.
(This conic is not the one defined by the condition $\theta^2=0$,
in fact $\big(a\omega_1^2+b\omega_\sigma^2+c\omega_2^2\big)^2=0$ in $H^8\big(J^2\big)$ iff $3b^2=-ac$.)
\end{proof}

\subsection[The union Z of the planes of Hodge classes PB\_K\^{}2]{The union $\boldsymbol{Z}$ of the planes of Hodge classes $\boldsymbol{{\bf P} B_K^2}$}\label{unz}

Using the basis for $B^2\big(J^2\big)$ given in Proposition~\ref{B2},
any element $\theta \in B^2$ in its complexification is a linear combination
\[
\theta =
x_0\omega_1^2+x_1\omega_1\omega_2+ x_2\omega_2^2 + x_3\omega_1\omega_\sigma +
x_4\omega_2\omega_\sigma + x_5\omega_\sigma^2\qquad\big(\!\!\in B^2\big),
\]
with $(x_0,\dots,x_5)\in{\bf C}^6$.
The following theorem identifies the union $Z$ of all the projective planes ${\bf P} B^2_K$
as $K$ varies over the Cartan subalgebras of $M_2({\bf C})$.
For dimension reasons we expect all these projective planes to be contained in a hypersurface.
It should be noticed that the surface $S=\operatorname{Sing}(Z)$ is \emph{not} the Veronese surface
of ${\bf P} B^1$.

\begin{Theorem} \label{thmz}
Let $Z$ be the Zariski closure in ${\bf P} B^2\cong{\bf P}^5$ of the union
$\cup_K {\bf P} B^2_K$, where the union is over all
diagonalizable $2$-dimensional Cartan subalgebras $K\hookrightarrow M_2({\bf C})$:
\[
Z := \overline{\cup_K {\bf P} B^2_K}\qquad\big(\!\!\subset {\bf P} B^2 \cong {\bf P}^5\big).
\]

Then $Z$ is a cubic fourfold in ${\bf P} B^2$ defined by
the determinant of a symmetric matrix $M_G$:
\[
Z = (G=0),\qquad
G = \det M_G,\qquad
M_G := \left(\begin{matrix}
2x_0&-x_3&x_1-4x_5\\
-x_3&x_1&x_4\\
x_1-4x_5&x_4&2x_2
\end{matrix}\right).
\]
The singular locus of $Z$ is a Veronese surface $S\;(\cong{\bf P}^2)$ and $Z$ is the secant variety of $S$.
A~parametrization of $S$ is given by
\begin{gather*}
\nu' \colon \ {\bf P}^2 \stackrel{\cong}{\longrightarrow} S := \operatorname{Sing}(Z),\\
(s:t:u)\longmapsto (x_0:\dots:x_5):=
\big(s^2: 2u^2:t^2: -2us: 2ut: \tfrac{1}{2}\big({-}st + u^2\big)\big).
\end{gather*}
The intersection of the space of Hodge classes $B^2_K$ of $\big(J^2,K\big)$ with $S$ is the conic
$C_K$ from Proposition~{\rm \ref{conicCK}}:
\[
{\bf P} B^2_K \cap S = C_K.
\]
\end{Theorem}

\begin{proof}
For a~general Cartan subalgebra $K$ we may assume that its eigenvectors
are of the form $v_\pm=(\alpha_\pm,1)\in{\bf C}^2$.
These eigenvectors determine the classes $\omega_\pm,\omega_K\in B^1$ as in
Lemma \ref{lemKB}.
Any element in $B^2_K$ can then be written as
\[
\omega(\alpha_+,\alpha_-,r,s,t) := r\omega_+^2 + s\omega^2_K + t\omega^2_-\qquad
\big(\!\!\in B_K^2\subset B^2\big),
\]
for a certain $(r,s,t)\in {\bf C}^3$. So the problem is to determine the image of ${\bf C}^5$,
with coordinates $\alpha_+$, $\alpha_-$, $r$, $s$, $t$, in ${\bf P} B^2$ under this map.
We used a computer to find the polynomial $G$ defining~$Z$ and the singular locus of $Z$.
(The `difference' between the Veronese surface of ${\bf P} B_1$ in ${\bf P} B^2$, which is
$\big\{\omega^2\in{\bf P} B^2\colon \omega\in B_1\big\}$ and which lies in~$Z$,
and $S$ is only in the component along $V_0$ in the decomposition
$B^2\big(J^2\big)=V_4\oplus V_0$ in Proposition~\ref{GLB}.)
Choosing $K$ to be the diagonal Cartan subalgebra as in the proof of Proposition~\ref{conicCK},
the plane ${\bf P} B^2_K$ is defined by $x_1=x_3=x_4=0$. The $2\times 2$ minors of~$M_G$,
which define~$S$, restrict to multiples of $4x_5^2-x_0x_2$, which is thus the equation of~$C_K$.
Using the ${\rm SL}_2$-action, the same result then holds for any $K$.

More intrinsically, the second Veronese map $\nu_2\colon {\bf P} B^1\rightarrow{\bf P} B^2$, with
$\nu_2(\omega)=\omega^2$, maps the invariant conic $C_2\subset{\bf P} B^1$ to
a quartic rational normal curve $C_4=\nu_2(C_2)$ which is the intersection of the Veronese surface
$\nu_2\big({\bf P} B^1\big)$ with a hyperplane $H\subset{\bf P} B^2$ (this hyperplane is the projectivization of
the subrepresentation~$V_4$).
The plane ${\bf P} B^2_K$ is spanned by $\omega_K^2$, $\omega_+^2$, $\omega_-^2$ where
$\omega_+,\omega_-\in C_2$. Hence ${\bf P} B^2_K$ intersects $H$ in the line spanned by
$\omega_+^2$, $\omega_-^2$. Thus $Z\cap H$ is the secant variety of~$C_4$, which is a cubic
(determinantal) threefold
(cf.\ \cite[Proposition~9.7]{Harris}). Therefore $Z$ is a cubic fourfold.

Since the planes ${\bf P} B^2_K$ intersect in points, the variety $Z$ is singular
and $\dim\operatorname{Sing}(Z)\geq 2$.
Any secant line of $\operatorname{Sing}(Z)$ intersects $Z$ with multiplicity at least four,
hence it lies in $Z$. Thus $Z$ is the secant variety of $\operatorname{Sing}(Z)$. The intersection
of $\operatorname{Sing}(Z)$ with $H$ is the curve~$C_4$, hence $\operatorname{Sing}(Z)$ is a surface of degree four
in ${\bf P}^5$. Since $Z$ contains the Veronese surface of~${\bf P} B^1$, it cannot be the secant variety
of a rational normal scroll and thus it is the secant variety of a Veronese surface by
\cite[Proposition, p.~525]{GH}.
\end{proof}

\begin{Remark} \label{linesLK}
For a Cartan subalgebra $K$, the conic $C_K=S\cap {\bf P} B^2_K$
is a conic in the Veronese surface $S=\nu'\big({\bf P}^2\big)$
and thus $C_K=\nu'(L_K)$ for a line $L_K\subset {\bf P}^2$. If $K'$ is another Cartan subalgebra,
then the lines $L_K$ and $L_{K'}$ intersect in a point $p\in {\bf P}^2$. Then
$\nu'(p)\in C_K\cap C_{K'}\subset {\bf P} B^2_K\cap {\bf P} B^2_{K'}$. Thus we find another proof, besides
Proposition~\ref{conicCK}, of the fact that any two planes $B^2_K$ intersect,
see also Section~\ref{cayleyclasses}.
\end{Remark}

\subsection[The rational curve C\_4 in the two Veronese surfaces]{The rational curve $\boldsymbol{C_4}$ in the two Veronese surfaces}

In ${\bf P} B^2$ there are two `natural' Veronese surfaces.
One is the image of $\nu_2\colon {\bf P} B^1\rightarrow {\bf P} \operatorname{Sym}^2\big(B^1\big)\allowbreak ={\bf P} B^2$ and the other is
the singular locus~$S$ of $Z$ which is the image of a~Veronese map $\nu'\colon {\bf P}^2\allowbreak \rightarrow {\bf P} B^2$.
The intersection of these two surfaces turns out to be a smooth rational
curve which we denote by~$C_4$.
Given a general point $c\in \cup_K {\bf P} B^2_K\subset Z$, this curve allows us to determine
the Cartan algebra $K$ such that $c\in B^2_K$, see Section~\ref{findB2p}.
We apply this in Section~\ref{cycles}.

\begin{Lemma}\label{quarticc4}\label{C4B2p}
The intersection of the two Veronese surfaces $\nu_2\big({\bf P} B^1\big)$ and $\nu'\big({\bf P}^2\big)$
is a smooth rational normal curve~$C_4$ of degree four
\[
C_4 = \nu_2\big({\bf P} B^1\big) \cap \nu'\big({\bf P}^2\big).
\]
The curve $C_4$ is the image under $\nu_2$ and $\nu'$ of the following conics:
\[
C_4 = \nu_2(C_2) = \nu'\big(\big\{(s:t:u)\in{\bf P}^2\colon u^2=st \big\}\big),
\]
where $C_2\subset{\bf P} B^1$ is the invariant conic.
The intersection of $C_4$ and ${\bf P} B^2_K$, for a Cartan subalgebra $K\subset M_2({\bf C})$,
consists of the two eigenvectors of $K^\times$ in ${\bf P} B^2_K$:
\[
C_4 \cap {\bf P} B^2_K = \big\{\omega_+^2, \omega_-^2 \big\}.
\]
\end{Lemma}

\begin{proof}The invariant conic $C_2$ can be parametrised by
$\omega=(xe_++ye_-)^2=x^2\omega_1+xy\omega_\sigma+y^2\omega_2\in {\bf P} B^1$
with $(x,y)\in{\bf C}^2$. Then $\nu_2(C_2)$ is a rational normal quartic curve
parametrised by the $\omega^2\in{\bf P} B^2$.
Similarly, the conic defined by
$u^2=st$ in ${\bf P}^2$ can be parametrised by $(s:t:u)=\big(x^2:y^2:xy\big)$.
One verifies that its image under~$\nu'$ is $\nu_2(C_2)$,
hence we found a curve $C_4\subset \nu_2(B^1)\cap \nu'\big({\bf P}^2\big)$.
A computer verified that there are no other points in the intersection.

From Theorem~\ref{thmz},
we know that the intersection of ${\bf P} B^2_K$ with the singular locus $S=\nu'\big({\bf P}^2\big)$ of $Z$
is the conic $C_K\subset {\bf P} B^2_K$ defined in Proposition~\ref{conicCK}.
The group $K^\times$ fixes $B^2_K$ and
has three orbits on $C_4=\nu_2(C_2)=\nu_4({\bf P} V_1)$, the fourth Veronese map,
they are the eigenvectors $\omega_\pm^2=\nu_4(v_\pm)$ and their complement.
The eigenvectors are in $C_4\cap {\bf P} B^2_K$,
but since $C_4$ spans a~${\bf P}^4$, the third orbit cannot intersect~${\bf P} B^2_K$.
\end{proof}

\begin{Remark}\label{remveroneses}
We discuss another aspect of the relation between the Veronese surfaces $\nu_2({\bf P} B^1)$,
with $\nu_2(\theta):=\theta^2\in {\bf P} B^2$, and $S=\nu'\big({\bf P}^2\big)=\operatorname{Sing}(Z)$.
For a Cartan subalgebra $K$, the plane of Hodge classes ${\bf P} B^2_K$ is the span of the
three points $\omega_+^2$, $\omega_-^2$ and $\omega_K^2$ in $\nu_2\big({\bf P} B^1\big)$. The closure of
the union of these planes is the cubic fourfold $Z$ which turned out to have another Veronese surface
$S$ as singular locus.
We showed that~$Z$ is the secant variety of~$S$
and that the intersection of $\nu_2\big({\bf P} B^1\big)$ and~$S$ is the quartic normal curve
$C_4=\nu_4({\bf P} V_1)$.

One can recover the Veronese surface $\nu_2\big({\bf P} B^1\big)$ from the secant variety $Z$ of $S$ and the
curve~$C_4$ as follows: the varieties ${\bf P} B^2\supset Z\supset S$
parametrize the conics, those of rank at most two and those of rank one in a ${\bf P}^2$.
The duality map $g\colon {\bf P} B^2\rightarrow {\bf P}^5$ is the birational map defined by
duality of conics. It sends the symmetric $3\times 3$ matrix defining the conic to its adjoint and
the base locus of $g$ is $S$.
The image of $Z$ is a Veronese surface $S^\vee\subset{\bf P}^5$. The image of a rank~2 conic,
so the union of two distinct lines in ${\bf P}^2$, is the intersection point of these two lines in ${\bf P}^2$.
The fiber of $g_{|Z-S}$ over $p$ consists of all (unordered) pairs of distinct lines
which intersect in $p$, so it is the complement of the diagonal in $\operatorname{Sym}^2{\bf P}^1\cong {\bf P}^2$.
The closure of this fiber is a plane ${\bf P}^2_p$ in ${\bf P} B^2$ which intersects~$S$ along a conic~$C_p$
that parametrizes the double lines that contain~$p$.

Recall that also the plane ${\bf P} B^2_K$ intersects $S$ along a conic $C_K$. As a plane intersecting
$S$ along a conic is the span of $\nu'(L)$ for a line $L$ in ${\bf P}^2$, there is only
a two-dimensional family of such planes and thus the general ${\bf P}^2_p$ is a ${\bf P} B^2_K$.

To recover $\nu_2\big({\bf P} B^1\big)$ we recall that ${\bf P} B^2_K$ intersects $C_4$ in the two points $\omega_\pm^2$.
By the proof of Proposition~\ref{conicCK},
$C_K$ is the conic in ${\bf P} B^2_K={\bf P}\big\{a\omega_+^2+b\omega_K^2+c\omega_-^2\big\}$ defined by $4b^2=ac$.
Thus the tangent lines to $C_K$ in the two points $\omega_\pm^2$ intersect in~$\omega_K^2$.
As any point in $\nu_2\big({\bf P} B^1\big)$, not on $C_4$, is a $\omega_K^2$ for a suitable~$K$,
we recover this Veronese surface from $S$ and $C_4$ as follows: in ${\bf P}^2_p$ we consider
the conic $C_p=S\cap {\bf P}^2_p$ and the two points $p_+,p_-\in C_p\cap C_4$. Let
$q_p\in {\bf P}^2_p$ be the point of intersection of the tangent lines to $C_p$ in $p_\pm$.
The image of the map $S^\vee\rightarrow Z\subset {\bf P}^5$, $p\mapsto q_p$
is the Veronese surface $\nu_2\big({\bf P} B^1\big)$. This map is a section of the duality map.
\end{Remark}

\subsection[Finding K]{Finding $\boldsymbol{K}$} \label{findB2p}

Given a codimension two Hodge class $c\in \cup {\bf P} B^2_K\subset Z\subset {\bf P} B^2$,
we want to determine a Cartan subalgebra $K$ such that $c\in {\bf P} B^2_K$. This $K$ will be
unique if $c$ is general. The case where $K$ is not unique is discussed in the next section
and occurs only if $c\in S$, the singular locus of $Z$ (see Proposition \ref{conicCK} and
Theorem \ref{thmz}).

As in Remark \ref{remveroneses},
a point $c\in Z-S$ corresponds to a reducible conic $L\cup M$ in a ${\bf P}^2$, which is the dual of the
domain of $\nu'$,
whereas a point $q\in S=\nu'({\bf P}^2)$ corresponds to a double line in this ${\bf P}^2$.
The pencil $sc+tq$ lies in $Z$ iff all conics in it are reducible,
so iff the double line which corresponds to $q$ passes through the point of intersection of $L$ and $M$.
There are two lines passing through this point that are tangent to the conic whose (double) tangent lines
parametrize~$C_4$. Thus for a given $c\in Z-S$ we find two points $q$ on $C_4$ such that
the lines $\langle c,q\rangle $ are contained in~$Z$.\looseness=-1

From Lemma~\ref{quarticc4}, we know that there are exactly two points
$\omega^2_\pm\in C_4$ in any ${\bf P} B^2_K$ and these points in fact determine~$K$.
In particular the two lines spanned by~$c$ and $\omega^2_\pm$
lie in ${\bf P} B^2_K$ and thus lie in $Z$.
We therefore consider the following procedure.

Let $q(x,y)$, for $(x,y)\in{\bf C}^2$, be the parametrization of $C_4$ given by the composition
${\bf P} V_1\rightarrow {\bf P} \operatorname{Sym}^2(V_1)\cong {\bf P} B^1\stackrel{\nu_2}{\rightarrow} {\bf P} B^2$:
\[
q\colon \ {\bf P} (V_1\otimes_{\bf Q}{\bf C}) \stackrel{\cong}{\longrightarrow} C_4,\qquad q(xe_++ye_-)=
\big(x^2\omega_1+xy\omega_\sigma+y^2\omega_2\big)^2.
\]
The line spanned by $c$ and $q(x,y)$ is parametrized by $sc+tq(x,y)$.
As $C_4\subset S$, the singular locus of $Z$, and $Z$ is defined by the cubic polynomial $G=0$,
we must have
\[
G(sc+tq(x,y))=f(x,y)^2s^2t,
\]
for some polynomial $f\in{\bf C}[x,y]$, homogeneous of degree $2$.
In case $f(x,y)=0$, the line $sc+tq(x,y)$ lies in $Z$.
Thus we conclude that for these $(x,y)$ the points
$q(x,y)=(xe_++ye_-)^4$ are the points~$\omega_\pm^2$. Hence $K$ is determined by the
two eigenvectors
$xe_++ye_-\in V_1\otimes_{\bf Q}{\bf C}$ where the points $(x:y)$ are the zeroes of~$f$.

The next theorem summarizes the main results we obtained.

\begin{Theorem} \label{thmsum}
A Hodge class in $Z-\operatorname{Sing}(Z)$ deforms in at most one four-dimensional family
of abelian fourfolds of Weil type, whereas a Hodge class in~$\operatorname{Sing}(Z)$ may deform in infinitely many
four-dimensional families of Weil type with distinct imaginary quadratic fields.
\end{Theorem}

\begin{proof}
For $c\in \cup {\bf P} B^2_K\subset Z$, $c\not\in S=\operatorname{Sing}(Z)$, there is a unique Cartan subalgebra $K$
such that $c\in{\bf P} B^2_K$, see Section~\ref{findB2p}. If $K$ is the complexification of an imaginary quadratic
field embedded in $M_2({\bf Q})$, then $c$ deforms in a~family of Weil type with that field.
If $c\in S$ then $c$ lies in infinitely many ${\bf P} B^2_K$, see also Section~\ref{cayleyclasses}.
\end{proof}

\section{The Cayley classes}\label{cayleyclasses}

\subsection[The Hodge classes in Sing(Z)]{The Hodge classes in $\boldsymbol{\operatorname{Sing}(Z)}$}

Combining Theorem~\ref{thmz} and Remark~\ref{linesLK} we conclude that any point
$c\in S=\operatorname{Sing}(Z)$ lies in~${\bf P} B^2_K$ for a one-dimensional family of
Cartan subalgebras $K$. Such classes, which we will call Cayley classes following~\cite{Markman},
are thus of special interest.

For a fixed Cayley class $c\in S$ we want to find all Cartan subalgebras $K$ such that $c\in B^2_K$.
For this it is convenient to use the parametrization $\nu'\colon {\bf P}^2\rightarrow S$ given in Theorem~\ref{thmz}. The plane of Hodge classes ${\bf P} B^2_K$ intersects $S$ along a conic~$C_K$
(see Proposition~\ref{conicCK})
and the inverse image $L_K:=\nu^{-1}(C_K)$ is a line in ${\bf P}^2$ which is easily computed:
\[
C_K = \nu'(L_K),\qquad L_K\colon \ s - \alpha_+\alpha_- t + (\alpha_+ + \alpha_-)u = 0
\qquad\big(\!\!\subset {\bf P}^2\big),
\]
where $v_\pm=(\alpha_\pm,1)$ are the eigenvectors of $K$. So if $p\in{\bf P}^2$ is such that
$c=\nu'(p)$ then $c\in B^2_K$ iff $p\in L_K$.

It is not always the case that a Cayley class in the ${\bf Q}$-vector space $B^2\big(J^2\big)$
underlying $B^2$ lies in a $B^2_K$, where $K$ is an imaginary quadratic field (rather than a
Cartan subalgebra of $M_2({\bf C})$). Taking for example the point $p=(n:1:0)\in {\bf P}^2$ we find that
that it lies in $L_K$ iff $\alpha_+\alpha_-=n$ (cf.\ Proposition~\ref{propmarkman}),
but for an imaginary quadratic field we must have
$\alpha_+=\overline{\alpha_-}\in K$ and then $\alpha_+\alpha_->0$. So for $n<0$ the Cayley class
$\nu'(p)$ does not deform to a Hodge class in any family of Weil type.

\begin{Example} We consider a Hodge class $C_\phi\in {\bf P} B^2$ which lies
in $S$:
\[
C_\phi = \omega_\sigma^2 + 4\omega_1\omega_2 = (0:1:0:0:0:4) = \nu'((0:0:1)).
\]
The corresponding Cartan subalgebras are thus the $K\subset M_2({\bf C})$,
determined by $v_\pm=(\alpha_\pm,1)$, such that the point
$p=(s:t:u)=(0:0:1)$ lies on the line $L_K$,
equivalently, $\alpha_++\alpha_-=0$.

In particular, for any $d>0$, the Cartan algebra defined by
the embedding $\phi\colon K={\bf Q}\big(\sqrt{-d}\big)\allowbreak \hookrightarrow M_2({\bf Q})$ given in Section~\ref{KEnd},
which is determined by the eigenvectors $v_\pm=(\alpha_\pm,1)$ with $\alpha_\pm=\pm\sqrt{-d}$,
has this Cayley class in $B^2_K$. Thus there are infinitely many, four-dimensional,
families of abelian varieties of Weil type with the property that $C_\phi$ lies
in the space of their Hodge classes.
\end{Example}

\subsection{Markman's Cayley class}
Markman in \cite{Markman} proves the Hodge conjecture
for all abelian fourfolds of Weil type with trivial discriminant.
For a fixed integer $n\geq 2$ he considers a
five-dimensional family of four-dimensional complex tori $T$ with
$H^1(T,{\bf Z})=V$, where $V:=H^1(X,{\bf Z})\oplus H^1\big(\hat{X},{\bf Z}\big)$ for an
abelian surface~$X$ and~$\hat{X}$ is the dual abelian surface.
The complex structure on~$T$ is determined by the subspace~$H^{1,0}(T)$
in the complexification of~$V$.

He shows that a certain class $C_n\in \wedge^4V= H^4(T,{\bf Z})$ remains of Hodge type $(2,2)$
for all members of the family. Moreover $C_n$ is shown to be
the second Chern class of a complex vector bundle ${\mathcal E}_T$ over $T$
using deformation theory in the hyperk\"ahler context.
The general torus in the family is not an abelian variety, but for each
imaginary quadratic field $K$ there is a four-dimensional subfamily of
abelian fourfolds of Weil type, with trivial discriminant, with this field.
Thus $C_n$ is algebraic, being a Chern class. Using the $K^\times$-action on $B^2_K$
one finds that all classes in $B^2_K$ are algebraic for any of these fourfolds of Weil type.

We will now assume that $X=J$, and that $\omega_J$ defines a principal polarization on~$J$.
Then we may identify $J^2$ with $X\times\widehat{X}$ and $H^1\big(J^2,{\bf Z}\big)$ with~$V$.
The four-dimensional families of Weil type will in general not
contain the three-dimensional subfamilies of
tori $T$ with $H^{1,0}(T)=H^{1,0}(J)\times H^{1,0}(J)$, where $J$ varies over the principally
polarized abelian surfaces.
They will often intersect this family in codimension one
and for the corresponding~$J$'s one has $\operatorname{End}_{\bf Q}(J)\neq {\bf Q}$, see Remark~\ref{cayT}.
The same is true for the similar families considered by O'Grady \cite[Theorem~5.1]{O'G}.

The following proposition identifies the cases in which the four-dimensional families
do contain all such decomposable abelian fourfolds of Weil type and their deformations.
It shows that Markman's Cayley classes~$C_n$ are indeed Cayley classes
as defined above. Moreover, it identifies
the corresponding imaginary quadratic fields and their embeddings into $\operatorname{End}\big(J^2\big)$.
The proposition shows that there must exist an
$\alpha_+\in K$ with conjugate $\alpha_-:=\overline{\alpha_+}$ and $\operatorname{Nm}(\alpha_+)=n$,
which is quite restrictive (for example if $p\equiv 3 \mod 4$ is a prime number then
there is no $\alpha\in {\bf Q}\big(\sqrt{-1}\big)$ with $\alpha\bar{\alpha}=p$).
This reflects the fact that the N\'eron Severi group of $J$ must have
rank two in the other cases.

\begin{Proposition} \label{propmarkman}
The Cayley class in {\rm \cite[Proposition~11.2]{Markman}} can also be written as
\[
C_n = n\big({-}n\omega_\sigma^2 + 2n^2\omega_1^2 + 2\omega_2^2\big)
\]
and $C_n\in \operatorname{Sing}(Z)$.
For any Cartan subalgebra $K\hookrightarrow M_2({\bf Q})$
with eigenvectors $v_\pm=(\alpha_\pm,1)$ such that
$\alpha_+\alpha_-=n$ we have $C_n\in B^2_K$. In fact
\[
(\alpha_+ - \alpha_-)^2C_{n} = n\big({-}\alpha_+\alpha_-\omega_K^2
 + 2\alpha_-^2\omega_+^2+2\alpha_+^2\omega_-^2\big)\qquad\big(\!\!\in B^2_K\big).
\]
\end{Proposition}

\begin{proof}The Cayley class from \cite[Proposition~11.2]{Markman} is
\[
C_n := n\big({-}nc_1({\mathcal P})^2 + 4n^2\pi_1^*[pt_X] + 4\pi_2^*[pt_{\hat{X}}]\big),
\]
where ${\mathcal P}$ is the normalized Poincar\'e bundle over $X\times\widehat{X}$.
Under the isomorphism $J^2\rightarrow X\times \widehat{X}$ we have $\omega_1^2=2\pi_1^*[pt_X]$
and $\omega_2^2=2\pi_2^*[pt_{\hat{X}}]$.
The isomorphism $J=X\rightarrow \widehat{X}$ induces the map on lattices
$\Lambda=H_1(J,{\bf Z})\rightarrow \widehat{\Lambda}=H_1\big(\widehat{X},{\bf Z}\big)$,
$\lambda'\mapsto \omega_J(-,\lambda')$.
The Chern class $c_1({\mathcal P})$
is the 2-form on $\Lambda\times\widehat{\Lambda}$ defined by
$((\lambda,l),(\mu,m)):=m(\lambda)-l(\mu)$.
One finds that $c_1({\mathcal P})$ pulls back to $\omega_\sigma$ and that the first expression for
$C_n$ holds.

To find all Cartan subalgebras $K\subset M_2({\bf Q})$ for which $C_n$ is a Hodge class in $B^2_K$,
it is convenient to identify the point in ${\bf P}^2$ which maps to $C_n$ under the Veronese map $v'$:
\[
C_n = n\big({-}n\omega_\sigma^2 + 2n^2\omega_1^2 + 2\omega_2^2\big)
 = \big(2n^2:0:2:0:0:-n\big) = \nu'((n:1:0)) \in S = \operatorname{Sing}(Z).
\]
Thus $C_n\in B^2_K$ iff $(s:t:u)=(n:1:0)\in L_K$, which is equivalent to $\alpha_+\alpha_-=n$.
The last equality is easy to verify.
\end{proof}

\begin{Remark}\label{cayT} We discuss the fact that in Proposition~\ref{propmarkman} one doesn't find
all imaginary quadratic fields.
With the notation from \cite[Section~12]{Markman}, the five-dimensional family is pa\-ra\-met\-rized
by $\Omega_{w^\perp}$ and $w\in S^+$ determines the class $C_n$.
The four-dimensional subfamilies are parametrized by $\Omega_{\{w,h\}^\perp}$ and
consist of the four dimensional tori ${\mathcal T}_\ell$ with
$\ell\in \Omega_{\{w,h\}^\perp}$, in particular $(h,\ell)=0$ for a class $h\in w^\perp\cap S^+$ with
$(h,h)_{S^+}<0$ (cf.\ \cite[formula~(12.6), Corollary~12.9]{Markman}). The imaginary quadratic field $K\subset \operatorname{End}_{\bf Q}(T_\ell)$ is generated by an endomorphism $\Theta'_h\in \operatorname{End}({\mathcal T}_\ell)$ \cite[Lemma 12.5]{Markman}.
One can find such $h$ for which not all periods $\ell\in \Omega_{w^\perp}$ in the three-dimensional family parametrizing the abelian fourfolds ${\mathcal T}_\ell\cong J\times J$ will satisfy $(h,\ell)=0$.
For such an $h$ and any $\ell\in\Omega_{\{w,h\}^\perp}$ such that ${\mathcal T}_\ell\cong J\times J$,
the endomorphism $\Theta'_h$ of $J\times J$ cannot be in $M_2({\bf Q})$,
since it does not deform to all selfproducts.
Therefore $\operatorname{End}_{\bf Q}(J)\neq {\bf Q}$. From this one can deduce that the Picard number of $J$ cannot be one.
\end{Remark}

\section{Applications to algebraic cycles}\label{cycles}

\subsection{Jacobians of genus two curves}\label{jacs}
We consider two cases, \cite{schoen} and \cite{vG}, where explicit Weil classes
on decomposable fourfolds $A=J\times J$ are given in order to illustrate
our results.

In both cases, the abelian surface $J$ is the Jacobian of a general genus two curve $C$.
After fixing a base point on $C$,
the Abel--Jacobi image of $C\hookrightarrow J$,
which we denote by $C$ again,
is a divisor which defines a principal polarization.
Using the notation from Section~\ref{JJ},
$\omega_J\in H^2(J,{\bf Q})$ is the class of $C\subset J$.
We now consider some codimension two classes defined by surfaces in~$J^2$.
Recall that translations act trivially on the cohomology of abelian varieties.

The diagonal $S_1:=\Delta=\{(x,x)\colon x\in J\times J\}$ is a codimension two subvariety.
It is the inverse image of the point $0\in J$ along the difference map:
\[
S_1 := \Delta := \delta^{-1}0,\qquad
\delta\colon \ J\times J \longrightarrow J,\qquad(x,y)\longmapsto x-y.
\]
Its class $[S_1]$ is thus the pull-back along $\delta$ of the class of a point on $J$.
Since the intersection of $C$ with a general translate consists
of two points (in fact $\omega_J^2=2{\rm d} x_1\wedge\dots \wedge{\rm d} y_2$),
we find $2[S_1]=\delta^*\big(\omega_J^2\big)$ and hence
\[
[S_1] = [\Delta] = \tfrac{1}{2} \delta^*\big(\omega_J^2\big) =
\tfrac{1}{2} (\omega_1+\omega_2-\omega_\sigma)^2\qquad\big(\!\!\in B^2\big(J^2\big)\big).
\]

The surface $S_2:=C\times C\subset J\times J$ is the intersection of the divisors
$C\times J$ and $J\times C$, these divisors have classes $\omega_1$, $\omega_2$ respectively and thus
\[
[S_2] = [C\times C] = \omega_1\omega_2\qquad\big(\!\!\in B^2\big(J^2\big)\big).
\]

\subsection{The class of the Schoen surface}\label{cc}
As in \cite{schoen}, we consider the 2-cycle $S:=S_1+S_2$ in $J\times J$
with $S_1$, $S_2$ as in Section \ref{jacs}. We will use our results to show that its
class in $H^4\big(J^2,{\bf Q}\big)$ lies in a unique $B^2_K$.
We also determine $K$ as well as the type of the (essentially
unique) polarization $\omega_K$ of Weil type. The class of~$S$ is given by
\[
2[S] = 2([S_1]+[S_2]) = \omega_1^2 + 4\omega_1\omega_2 + \omega_2^2
 - 2\omega_1\omega_\sigma - 2\omega_2\omega_\sigma + \omega_\sigma^2.
\]

It is easy to check that $[S]$ lies on the cubic fourfold $Z=\overline{\cup B^2_K}$.
To find the subalgebra $K\subset M_2({\bf Q})$ such that $[S]\in B^2_K$
we consider the intersection of the line in ${\bf P}^5$
spanned by $c:=2[S]$ and a general point $q(x,y)=(xe_++ye_-)^4 \in C_4$ as in Section \ref{findB2p}.
A computation shows that
\[
G(sc+tq)=8\big(x^2+xy+y^2\big)^2s^2t,\qquad\text{hence}\quad
x^2+xy+y^2 = (x-\alpha_+y)(x-\alpha_-y),
\]
with $\alpha_\pm=\big({-}1\pm\sqrt{-3}\big)/2$. The Cartan subalgebra~$K$ is
determined by the eigenvectors $v_{\pm}=(\alpha_\pm,1)$ and thus the
field is $K={\bf Q}\big(\sqrt{-3}\big)$. The eigenvectors also determine the embedding
$K\hookrightarrow M_2({\bf Q})$:
\[
K = {\bf Q}\big(\sqrt{-3}\big) \hookrightarrow M_2({\bf Q}),\qquad
\frac{-1+\sqrt{-3}}{2} \longmapsto \left(\begin{matrix}-1&-1\\ \hphantom{-}1&\hphantom{-}0\end{matrix}\right).
\]

With this embedding of $K$, the Hodge class $[S]$ lies in $B^2_K$ and
will therefore deform to a~(unique) family of abelian fourfolds of
Weil type with this field~$K$.

The polarization $\omega_K$ is (up to scalar multiple)
equal to $\omega_K$ and we choose (cf.\ Lemma~\ref{lemKB})
\[
\omega_K = 2\omega_1 - \omega_\sigma + 2\omega_2,
\qquad\operatorname{Pfaff}(\omega_K) = (1-2\cdot 2)^2 = 9.
\]
Notice that in \cite[proof of Lemma~10.4]{schoen} one finds an ample line bundle
${\mathcal L}$ on a general deformation of Weil type which has
$h^0({\mathcal L})=\operatorname{Pfaff}(c_1({\mathcal L}))=9$,
our computation is consistent with this.
A further computation shows that the elementary divisors of the alternating form
defined by $\omega_K$ on ${\bf Z}^8$ are $(1,1,3,3)$.

In \cite{schoen}, Schoen proved that actually the surface $S_1\cup S_2$ in $J^2$ deforms
in a family of abelian varieties of Weil type with field ${\bf Q}\big(\sqrt{-3}\big)$. These deformations
were further studied in \cite{CMLR} and~\cite{RRS}.

\subsection[The field Q(i)]{The field $\boldsymbol{{\bf Q}(i)}$}\label{qi}
In \cite{vG} one finds a four-dimensional family of principally polarized
abelian fourfolds of Weil type with field $K={\bf Q}\big(\sqrt{-1}\big)$.
There is a rational dominant map from such a fourfold $A$
(in general with $16$ base points) $A\rightarrow Q$,
where $Q\subset {\bf P}^5$ is a smooth four-dimensional quadric.
The pull-back to~$A$ of a plane~${\bf P}^2\subset Q$ is shown to be an exceptional cycle $c$
(i.e., its class lies in $B^2(A)$ but is not a scalar multiple of $\omega_K^2$). The class of
this cycle was not completely determined in~\cite{vG}, but it can be easily done now.

In the proof of Theorem~3.7 of~\cite{vG} it is shown that, for some $a,b\in{\bf Z}$,
the class $c$ of the pull-back of a certain plane to $J\times J$ is
\begin{gather*}
c = [T]+a[S_1]+b[S_{-1}],\qquad
[T] = [(2C)\times(2C)]=4[S_2],\\
S_{-1} := \{(x,y)\in J\times J\colon x+y=0\}.
\end{gather*}
The class of $S_{-1}$ is obviously $(1,-1)^*[S_1]$, hence
\[
[S_{-1}]=\tfrac{1}{2}(\omega_1+\omega_2+\omega_\sigma)^2.
\]

To determine $a$, $b$ we recall from~\cite{vG} that the family contains fourfolds
$A=J\times J$, with~$J$ a~general abelian surface as before, the polarization
$\omega_K$ is the product
polarization $\omega_1+\omega_2$ and
the embedding of the field is
\[
K = {\bf Q}(\sqrt{-1}) \hookrightarrow M_2({\bf Q}),\qquad
\sqrt{-1} \longmapsto \left(\begin{matrix}0&-1\\1&\hphantom{-}0\end{matrix}\right).
\]
This embedding has the eigenvectors $v_\pm=(\mp {\rm i},1)$.
From Lemma~\ref{lemKB} one finds that $B^2_K$ is spanned by
\[
\omega_K^2 = 4(\omega_1+\omega_2)^2,\qquad \omega_+^2 = (\omega_1+{\rm i}\omega_\sigma-\omega_2)^2,\qquad
\omega_-^2 = (\omega_1-{\rm i}\omega_\sigma-\omega_2)^2.
\]

It is now easy to check that
\[
[T]+a[S_1]+b[S_{-1}] \in B^2_K\ \Longleftrightarrow \ a = b =1,
\]
thus we determined the class of the exceptional cycle: $c=[T]+[S_1]+[S_{-1}]$.

\section{The discriminant of the polarization}\label{secdis}

\subsection{The discriminant}
Let $(A,K,E)$ be a polarized abelian variety of Weil type with imaginary quadratic field
$K={\bf Q}\big(\sqrt{-d}\big)$ and $\dim A=n$.
The ${\bf Q}$-bilinear alternating form $E$ on $H_1(A,{\bf Q})$ satisfies
$E(kx,ky)=\operatorname{Nm}(k)E(x,y)$ for $x,y\in H_1(A,{\bf Q})$ and $k\in K$,
where $\operatorname{Nm}(k)=k\bar{k}\in {\bf Q}$ is the norm of $k$.
Then $H_1(A,{\bf Q})$ is a $K$-vector space and
$E$ defines a $K$-valued Hermitian form $H$ on this $K$-vector space by:
\[
H \colon \ H_1(A,{\bf Q})\times H_1(A,{\bf Q}) \longrightarrow K,\qquad
H(x,y) := E\big(x,\big(\sqrt{-d}\big)_*y\big) + \sqrt{-d}E(x,y).
\]
If $\Psi\in M_n(K)$ is the Hermitian matrix defining $H$ w.r.t.\ some $K$-basis of $H_1(A,{\bf Q})$
then $\det(\Psi)\in {\bf Q}^\times={\bf Q}-\{0\}$ and the class
of $\det(\Psi)\in {\bf Q}^\times/\operatorname{Nm}(K^\times)$ is independent of the choice of the basis,
it is called the discriminant of $(A,K,E)$.

\subsection{Trivial discriminant for selfproducts}\label{triv}
Let $A=J\times J$, for an abelian surface $J$ and let $K\subset \operatorname{End}_{\bf Q}(A)=M_2({\bf Q})$ be an
imaginary quadratic field.
Let $\omega_K\in H^2(A,{\bf Q})$ be the unique
(up to scalar multiple) $K$-invariant
polarization as in Lemma~\ref{lemKB} on the abelian fourfold of Weil type~$(A,K)$.
We show in Proposition \ref{triv dis} that the discriminant of
$(A,K,\omega_K)$ is trivial, and we briefly discuss more general decomposable
abelian fourfolds of Weil type in Remark~\ref{gendis}.

\begin{Proposition}\label{triv dis}
Let $J$ be a general abelian surface, let $A=J\times J$,
and let $(A,K,\omega_K)$ be an abelian fourfold of Weil type.
Then the discriminant of $(A,K,\omega_K)$ in ${\bf Q}^\times/\operatorname{Nm}(K^\times)$ is trivial.
\end{Proposition}

\begin{proof}
Recall the basis ${\rm d} t_i$, $i=1,\dots, 4$, of $H^1(J,{\bf Q})$ from Section~\ref{JJ}. Let
$f_1,\dots,f_4$ be the dual basis of $H_1(J,{\bf Q})$, so that
${\rm d} t_i(f_j)=\delta_{ij}$, where $\delta_{kl}$ is
Kronecker's delta.
Then $H_1(A,{\bf Q})=H_1(J,{\bf Q})^{\oplus 2}$ and, since $\operatorname{End}_{\bf Q}(J)={\bf Q}$, the vectors
$(f_1,0),\dots,(f_4,0)$ are a $K$-basis of $H_1(A,{\bf Q})$.

In view of the definitions given in Proposition~\ref{B2},
for $i,j=1,\dots,4$, we find the following table:
\begin{alignat*}{3}
& \omega_1((f_i,0),(f_j,0))=\omega_J(f_i,f_j),\qquad && \omega_1((f_i,0),(0,f_j))=0,&\\
& \omega_1((0,f_i),(0,f_j))=0,&&&\\
&\omega_2((f_i,0),(f_j,0))=0,\qquad &&\omega_2((f_i,0),(0,f_j))=0,&\\
& \omega_2((0,f_i),(0,f_j))=\omega_J(f_i,f_j),&&& \\
& \omega_\sigma((f_i,0),(f_j,0))=0,\qquad && \omega_\sigma((f_i,0),(0,f_j))=\omega_J(f_i,f_j),&\\
& \omega_\sigma((0,f_i),(0,f_j))=0,&&&
\end{alignat*}
here one computes (with $\sigma$ as in the proof of Proposition~\ref{B2}):
\[
\omega_\sigma((f_i,0),(0,f_j)) = \sigma^*\omega_J((f_i,0),(0,f_j))
\]
since $\sigma^*\omega_J=\omega_1+\omega_2+\omega_\sigma$ and
$\omega_1$, $\omega_2$ are zero on $((f_i,0),(0,f_j))$,
and finally $\sigma^*\omega_J((f_i,0),(0,f_j))\!=\!\omega_J(\sigma(f_i,0),\sigma(0,f_j))=\omega_J(f_i,f_j)$.

The action of $K\subset M_2({\bf Q})$ on $H_1(A,{\bf Q})=H_1(J,{\bf Q})\otimes V_1$ is determined by
$a,b\in{\bf Q}$ such that $\sqrt{-d}_*(f_i,0)=(af_i,bf_i)$,
for $i=1,\dots,4$.
The polarization is a linear combination
\[
\omega_K = a_1\omega_1+a_\sigma \omega_\sigma+a_2\omega_2,\qquad a_1,a_\sigma,a_2\in{\bf Q}.
\]
The Hermitian form $H$ is determined by its matrix $\Psi=(\Psi_{ij})$
on the $K$-basis of $H_1(A,{\bf Q})$ given above:
\[
\Psi_{ij} := \omega_K((f_i,0),(af_j,bf_j)) + \sqrt{-d}\omega_K((f_i,0),(f_j,0)).
\]
From the table we see that
\[\omega_K((f_i,0),(f_j,0))=a_1\omega_J(f_i,f_j),\qquad
\omega_K((f_i,0),(0,f_j))=a_\sigma \omega_J((f_i,0),(0,f_j))
\]
and thus
\[
\Psi_{ij} = \big(a_1\big(a+\sqrt{-d}\big)+ba_\sigma\big)\omega_J(f_i,f_j).
\]
Since $H$ is a Hermitian form, we have $\Psi_{ij}=\overline{\Psi_{ji}}$ and using
$\omega_J(f_i,f_j)=-\omega_J(f_j,f_i)$ we get $aa_1+ba_\sigma=0$.
Therefore $\Psi_{ij}=\big(a_1\sqrt{-d}\big)\omega_J(f_i,f_j)$. Since the $f_i$ are a symplectic basis
of $H_1(J,{\bf Q})$ we get $\det(\omega_J(f_i,f_j))=1$. Thus $\det(\Psi)=a_1^4d^2=
\operatorname{Nm}\big(a_1^2d\big)$ is the norm of an element in ${\bf Q}\subset K$, so the discriminant of
$(A,K,\omega_K)$ is trivial.
\end{proof}

\subsection{General discriminants} \label{gendis}
Whereas $J^2$, with $\operatorname{End}_{\bf Q}(J)={\bf Q}$, only has structures of Weil type with trivial
discriminant, for example the self product
$E^4$ of an elliptic curve $E$ with CM by $K$ admits polarizations with any discriminant.
In fact, if $\operatorname{End}_{\bf Q}(E)=K$, then the embedding
$K\hookrightarrow \operatorname{End}_{\bf Q}\big(E^4\big)=M_4(K)$ given by $k\mapsto \big(k,k,\bar{k},\bar{k}\big)$
defines an abelian fourfold of Weil type. Let $\omega_E$ be a polarization on~$E$
(it is unique up to scalar multiple). Let
$f$ be a $K$-basis of $H_1(E,{\bf Q})\cong K$ and let $b:=\omega_E\big(f,\sqrt{-d}f\big)\in{\bf Q}$.

For any $a\in{\bf Q}_{>0}$ the polarization $\omega_a\in\wedge^2H^1\big(E^4,{\bf Q}\big)$ on $E^4$ defined by
\[
\omega_a((x_1,\dots,x_4),(y_1,\dots,y_4)) = \omega_E(x_1,y_1) + \dots +
\omega_E(x_3,y_3) + a\omega_E(x_4,y_4),
\]
where $x_i,y_i\in H_1(E,{\bf Q})$, satisfies
$\omega_a(k\cdot x,k\cdot y)=k^2\bar{k}^2\omega_a(x,y)$.
The matrix of the associated Hermitian form on the $K$-basis $(f,0,0,0),\dots,(0,0,0,f)$
of $H_1\big(E^4,{\bf Q}\big)$ is given by $\operatorname{diag}(b,b,-b,-ab)$ and thus it has determinant $ab^4$.
Hence the discriminant of the polarized abelian variety of Weil type $\big(E^4,\omega_a,K\big)$ is~$a$.
The dimension of $B^2\big(E^4\big)$ is however much larger than $6$ and the study of the limits
of spaces of Weil classes is thus more complicated.

\begin{Remark} The Hodge conjecture for abelian fourfolds of Weil type with fields
${\bf Q}\big(\sqrt{-3}\big)$, ${\bf Q}\big(\sqrt{-1}\big)$ and any discriminant has been proven
in \cite{schoenhodge}, \cite{Sadd} and~\cite{Koike}, respectively.
The method for both results
is the same: one verifies the Hodge conjecture for a family of abelian six-folds of
Weil type with the same field and trivial discriminant. Then one specializes to a product
$J\times B$, where~$B$ is a general abelian fourfold of Weil type and the discriminant
of~$B$ can be arbitrary. A~variant of this method, using abelian varieties with
quaternionic multiplication, was proposed in~\cite{vGV} but thus far has not had any
applications.
\end{Remark}

\subsection*{Acknowledgements}
Discussions with E.~Markman, K.G.~O'Grady, F.~Russo and C.~Schoen were very helpful.
I~thank the referees for their comments and suggestions.

\pdfbookmark[1]{References}{ref}
\LastPageEnding

\end{document}